\documentclass{amsart}

\makeatletter
\def\subsection{\@startsection{subsection}{2}
  \z@{.5\linespacing\@plus.7\linespacing}{.5\linespacing}
  {\normalfont\bfseries}}
\makeatother

\makeatletter
\def\@defaultbiblabelstyle#1{[#1]}
\makeatother

\makeatletter
\def\@setauthors{
  \begingroup
  \def\thanks{\protect\thanks@warning}
  \trivlist
  \centering\footnotesize \@topsep30\p@\relax
  \advance\@topsep by -\baselineskip
  \item\relax
  \author@andify\authors
  \def\\{\protect\linebreak}
  \authors
  \ifx\@empty\contribs
  \else
    ,\penalty-3 \space \@setcontribs
    \@closetoccontribs
  \fi
  \endtrivlist
  \endgroup
}
\def\@settitle{\begin{center}
  \baselineskip14\p@\relax
    \bfseries
  \@title
  \end{center}
}
\makeatother

\usepackage{amssymb}
\usepackage{amsxtra}
\usepackage{mathtools}
\usepackage[colorlinks=true, citecolor=blue]{hyperref}
\usepackage{upref}
\usepackage{soul}

\newtheorem{theorem}{Theorem}[section]
\newtheorem{lemma}[theorem]{Lemma}
\newtheorem{proposition}[theorem]{Proposition}
\newtheorem{corollary}[theorem]{Corollary}

\newtheorem{question}[theorem]{Question}

\theoremstyle{definition}

\newtheorem{notation}[theorem]{Notation}

\theoremstyle{remark}
\newtheorem{remark}[theorem]{Remark}

\numberwithin{equation}{section}
\allowdisplaybreaks[4]

\begin{document}

\title[Weyl group symmetries]{Weyl group symmetries of the toric variety associated with Weyl chambers}

\author{Tao Gui}
\address{(Tao Gui) \newline \indent Beijing International Center for Mathematical Research, Peking University, No.\ 5 Yiheyuan Rd, Haidian District, Beijing 100871, China}
\email{guitao18(at)mails(dot)ucas(dot)ac(dot)cn}

\author{Hongsheng Hu}
\address{(Hongsheng Hu) \newline \indent School of Mathematics, Hunan University, Changsha 410082, China}
\email{huhongsheng(at)hnu(dot)edu(dot)cn}

\author{Minhua Liu}
\address{(Minhua Liu) \newline \indent School of Mathematical Sciences, Shanghai Jiao Tong University, 800 Dongchuan RD, Minhang District, Shanghai 200240, China}
\email{mliu8646(at)gmail(dot)com}



\begin{abstract}
For any crystallographic root system, let $W$ be the associated Weyl group, and let $\mathit{WP}$ be the weight polytope (also known as the $W$-permutohedron) associated with an arbitrary strongly dominant weight. The action of $W$ on $\mathit{WP}$ induces an action on the toric variety $X(\mathit{WP})$ associated with the normal fan of $\mathit{WP}$, and hence an action on the rational cohomology ring $H^*\left(X(\mathit{WP})\right)$. Let $P$ be the corresponding dominant weight polytope, which is a fundamental region of the $W$-action on $\mathit{WP}$. We give a type uniform algebraic proof that the fixed subring $H^*\left(X(\mathit{WP})\right)^{W}$ is isomorphic to the cohomology ring $H^*\left(X(P)\right)$ of the toric variety $X(P)$ associated with the normal fan of $P$. Notably, our proof applies to all finite (not necessarily crystallographic) Coxeter groups, answering a question of Horiguchi--Masuda--Shareshian--Song about noncrystallographic root systems.
\end{abstract}

\maketitle

\setcounter{tocdepth}{2}

\section{Introduction}

Let $\Phi$ be a (reduced finite) crystallographic root system of rank $r$ in the Euclidean space $E$ and $W$ be the corresponding Weyl group. Let $\lambda$ be any strongly dominant integral weight and $\mathit{WP}:=\operatorname{Conv} (W \lambda)$ be the convex hull of the $W$-orbit $W \lambda$ in $E$, known as the weight polytope or $W$-permutohedron associated with $\lambda$. The normal fan of $\mathit{WP}$ is the complete simplicial fan with cones given by Weyl chambers and their faces, hence it determines a (smooth) projective toric variety $X(\mathit{WP})$. The variety $X(\mathit{WP})$ can also be realized as the generic orbit closure of the maximal torus action on the corresponding flag variety \cite{klyachko1985orbits} and it is also a special kind of Hessenberg variety \cite{de1992hessenberg}. In type $A$, the corresponding variety is usually called the permutohedral variety and it is crucial in the recent development of the combinatorial algebraic geometry around matroids, see \cite{huh2014rota} and \cite{baker2018hodge} for instance.

The reflection action of $W$ on $E$ preserves $\mathit{WP}$ and so induces an action of $W$ on $X(\mathit{WP})$. The induced (graded) representation of $W$ on the rational cohomology ring $H^*\left(X(\mathit{WP})\right)$ has been well-studied, see \cite{Procesi90}, \cite{stembridge1994some}, \cite{dolgachev1994character} and \cite{lehrer2008rational}. Each of these papers provides a different perspective. 

On the other hand, one can consider the fundamental region $\mathit{WP} / W$ of the $W$-action on $\mathit{WP}$, which can be identified with the dominant weight polytope $P$, defined as 
\begin{equation*}
       P := \operatorname{Conv} (W \lambda) \cap \overline{C_+} \subset E, 
\end{equation*}
where $\overline{C_+}$ is the closure of the dominant Weyl chamber $C_+$. That is, $P$ is also a convex polytope obtained from intersecting the cone $\overline{C_+}$ with the convex set $\operatorname{Conv} (W \lambda)$. See \cite[fig.\ 1]{BGHpolytope} for some examples.
The polytope $P$ is studied in a previous work \cite{BGHpolytope}, where the authors showed that $P$ is also a rational (with respect to the root lattice) simple polytope and it is combinatorially equivalent to the $r$-dimensional cube. Therefore, its normal fan is also an complete simplicial fan which also determines a projective toric variety $X(P)$, which is usually not smooth (it is actually a toric orbifold). It natural to ask the following question (which is a special case of \cite[Question 8.3]{HMSS21-toric-orbifolds}).

\begin{question} \label{ques}
    Is the quotient $X(\mathit{WP}) / W$ isomorphic to $X(P)$?
\end{question}

In \cite{Blume15}, Blume studied $X(P)$ (in which he called it the toric orbifold associated to the Cartan matrix) and showed that the above question has positive answers in types $A_r$, $B_r$, and $C_r$ by investigating variant moduli stacks of pointed chains of $\mathbb{P}^1$ related to the Losev--Manin moduli spaces.

Note that we have the classical well-known result
$$
H^*(X / W ; \mathrm{k}) \cong H^*(X ; \mathrm{k})^W
$$
for any locally compact Hausdorff $W$-space for some finite group $W$ and for any field $\mathrm{k}$ of characteristic prime to the order of $W$. One purpose of this paper is to give a type uniform proof of the following theorem, which provides cohomological evidence of the above question.

\begin{theorem} \label{thm-main}
   Let $\Phi$ be a root system of any Lie type, then there exists an explicit graded ring isomorphism
   \begin{equation} \label{iso}
     H^*\left(X(P) ; \mathbb{Q}\right) \cong H^*\left(X(\mathit{WP}) ; \mathbb{Q}\right)^W.      
   \end{equation}
\end{theorem} 

The above theorem is proved for the classical types $A_r, B_r, C_r$, and $D_r$ case by case in \cite{HMSS21-toric-orbifolds} (actually, they proved a more general parabolic version), and it is proved for rank-$2$ root systems (moreover, for any reflection symmetry of them) in \cite{Song22}. Furthermore, it can be deduced for any root system from the fact that both sides of Eq.~\eqref{iso} are isomorphic to the rational cohomology of the corresponding so-called Peterson variety, see for example, \cite{abe2023peterson} and \cite{zbMATH07273426}. Actually, in a recent preprint \cite{gong2024homotopytypestoricorbifolds}, the author constructed a homotopy equivalence between the underlying topological spaces $X(\mathit{WP}) / W$ and $X(P)$ (and also for a more general parabolic version), hence it induces an isomorphism on their cohomologies. 
However, we do not know how to make this isomorphism explicit and do not know whether it agrees with the isomorphism constructed in the proof of Theorem \ref{thm-main}.
We emphasize that our aim is to provide a type uniform algebraic proof not depending on the classification of root systems and understand it in the world of toric varieties without introducing the Peterson varieties, which in fact works for noncrystallographic root systems (see Section \ref{sec-gen}). As in that case there is no variety to apply the geometric or topological technique, our algebraic point of view focusing on the rings with presentations provides a different useful perspective.

Our proof relies on an explicit and type uniform presentation of $H^*\left(X(P) ; \mathbb{Q}\right)$. This result follows from analyzing the normal fan of the dominant weight polytope $P$
and then applying Danilov's theorem on the presentation of the cohomology of arbitrary toric orbifolds \cite{danilov1978geometry}. 

\begin{theorem} \label{thm-pres-P}
    The torus equivariant cohomology ring $H^*_T(X(P) ; \mathbb{Q})$ has the following presentation 
    \begin{equation} 
      H^*_T\left(X(P) ; \mathbb{Q}\right)=\mathbb{Q}\left[x_1, \ldots, x_r, y_1, \ldots, y_r\right] /\mathcal{I},
    \end{equation}
    and the cohomology ring $H^*\left(X(P) ; \mathbb{Q}\right)$ has the following presentation
    \begin{equation} \label{eq-pres-P}
      H^*\left(X(P) ; \mathbb{Q}\right)=\mathbb{Q}\left[x_1, \ldots, x_r, y_1, \ldots, y_r\right] /\left(\mathcal{I}+\mathcal{J}\right),
    \end{equation}
    where $\mathcal{I}$ is the ideal generated by
    \[\{x_i y_i \mid i \in [r]\},\]
    and $\mathcal{J}$ is the ideal generated by 
    \[\left\{\sum_{i=1}^r\left\langle q, \varpi_i\spcheck \right\rangle x_i+\sum_{i=1}^r\left\langle q, -\beta_i \right\rangle y_i \Biggm\vert q \in \Lambda_r\right\}.\]
    Here $\Lambda_r$ denotes the root lattice and $\beta_i$ denotes the primitive vector in the coweight lattice $\Lambda_r^*$ in the same direction as $\alpha_i\spcheck$.
\end{theorem}

Using the above presentation, we define an explicit homomorphism from $H^*(X(P) ; \mathbb{Q})$ to $H^*(X(\mathit{WP}) ; \mathbb{Q})^W$, and prove that it is indeed an isomophism using results of Procesi \cite{Procesi90} and Lehrer \cite{lehrer2008rational}.

\begin{remark} \leavevmode
  \begin{enumerate}
      \item The presentations in Theorem \ref{thm-pres-P} are already known in \cite[section~6.5]{abe2023peterson}, although they only work with the fan and do not involve the dominant weight polytope. Here we point our that the fan in their paper is indeed the normal fan of a dominant weight polytope. We also note that they do not use the primitive vectors $\beta_i$'s. However, by replacing $y_i$ by a suitable multiple, it is easy to see that our presentation is equivalent to theirs.
      \item As mentioned, $H^*(X(P) ; \mathbb{Q})$ is isomorphic to the cohomology $H^*(\mathit{Pet} ; \mathbb{Q})$ of the corresponding Peterson variety $\mathit{Pet}$.
    In \cite{HHM15}, a presentation of $H^*(\mathit{Pet} ; \mathbb{Q})$ is given. 
    The relationship between this presentation and the one in \eqref{eq-pres-P} is discussed in \cite[section~6.5]{abe2023peterson}.
      \item As mentioned above, Question \ref{ques} is a special case of \cite[Question 8.3]{HMSS21-toric-orbifolds}: is the quotient $X(\mathit{WP}) / W_K$ isomorphic to $X(\mathit{WP}/W_K)$ for any parabolic subgroup $W_K \subset W$?
        It is desirable to consider its cohomology version, that is, Theorem \ref{thm-main} for parabolic subgroups.
        However, this might be more difficult. 
        In the full Weyl group case, we have the presentations in Theorem \ref{thm-pres-P} thanks to the result \cite{BGHpolytope} on the combinatorial structure of the polytope $P$ ($=\mathit{WP} / W$). 
        The combinatorics and the corresponding cohomology rings for parabolic cases are more complicated, which might be appropriately pursued in future work.
  \end{enumerate}
\end{remark}

The crystallographic condition on the root systems allows the construction of the toric varieties $X(P)$ and $X(\mathit{WP})$. However, it is well-known that one can construct the ``virtual cohomology ring'' for any complete simplicial fan which is a quotient of the corresponding Stanley--Reisner ring by a linear ideal, see \eqref{eq-def-R(P)} and \eqref{eq-def-R(WP)}. It is natural to ask whether Theorem \ref{thm-main} holds in this general setting, see \cite[Question 8.2]{HMSS21-toric-orbifolds}). Our algebraic proof of Theorem \ref{thm-main} actually works in this generality after minor modifications, see Section \ref{sec-gen}.

\subsection*{Acknowledgments}
The authors would like to thank Connor Simpson and Haozhi Zeng for useful discussions.
The authors are grateful to the anonymous referee for constructive suggestions which helped to improve this paper.
The second author is supported by the Fundamental Research Funds for the Central Universities (No.~531118010972).

\section{Preliminaries}

\subsection{Notations}

In this section, we collect our notations for later use. 

$E$: The $r$-dimensional Euclidean space where the root system $\Phi$ lives.

$E^*$: The dual space of $E$, where the coweight lattice lives.

$\Lambda_r$: The root lattice $\mathbb{Z} \Phi$.

$\Lambda_r^*$: The coweight lattice, which is dual to $\Lambda_r$.

$\langle \cdot, \cdot \rangle$: The natural pairing between $E$ and $E^*$, as well as between $\Lambda_r$ and $\Lambda_r^*$.

$\alpha_1, \dots, \alpha_r$: Simple roots, which form a basis of $\Lambda_r$.

$\varpi_1, \dots, \varpi_r$: fundamental weights.

$\alpha_1\spcheck, \dots, \alpha_r\spcheck$: Simple coroots.

$\varpi_1\spcheck, \dots, \varpi_r\spcheck$: Fundamental coweights, which form a basis of $\Lambda_r^*$.

$\beta_1, \dots, \beta_r$: $\beta_i$ denotes the primitive vector in $\Lambda_r^*$ in the same direction as $\alpha_i\spcheck$.

For a rational (with respect to the root lattice $\Lambda_r$) polytope $Y$ in $E$ and a facet $F$ of $Y$, let $\eta(F) \in \Lambda_r^*$ denote the outward primitive normal vector of $F$. 

$[r] := \{1,\dots, r\}$.

$e$: The identity element of the Weyl group $W$.

\subsection{Cohomology rings of toric orbifolds}

A toric variety is an algebraic variety containing an algebraic torus as an open dense subset, such that the action of the torus on itself extends to the whole variety. One can construct a normal toric variety from a fan, which is a collection of polyhedral rational cones closed under taking intersections and faces.

Let $Y$ be a rational simple convex polytope in $E$ with facets $F_1, \ldots, F_m$. Let $\eta(F_i) \in \Lambda_r^*$ be the outward primitive normal vector of $F_i$. The (outer) normal fan $\Sigma_Y$ of $Y$ consists of rational convex polyhedral cones in $E^*$ generated by those sets $\{\eta(F_{i_1}), \ldots, \eta(F_{i_k})\}$ for which the intersection $F_{i_1}\cap\cdots\cap F_{i_k}$ is nonempty. Note that $\Sigma_Y$ is a rational simplicial fan, that is, each $i$-dimensional cone is generated by $i$ rational vectors. Let $X(Y)$ be the complex projective toric variety associated with $\Sigma_Y$ of complex dimension $r$, which is a rationally smooth orbifold. See, for example, \cite[sections 2.3 and 5.1]{cox2024toric} for explicit definitions and constructions of $X(Y)$.

According to the Jurkiewicz--Danilov theorem, the homomorphism 
\[
\mathbb{Q}[x_1, \ldots, x_m]\to H^*(X(Y) ; \mathbb{Q})
\]
 sending $x_i$ to the class $[D_i]$ of the torus-invariant divisor associated with $F_i$ is surjective and induces an isomorphism of rings as follows. 

\begin{theorem} [{\cite[Theorem 10.8 and Remark 10.9]{danilov1978geometry}}] \label{thm-pres-coh}
    The cohomology ring $H^*(X(Y) ; \mathbb{Q})$ is given by
\[
H^*(X(Y) ; \mathbb{Q})\simeq \mathbb{Q}[x_1, \ldots, x_m]/(\mathcal{I}+\mathcal{J}),
\]
where each variable $x_1, \ldots, x_m$ is of degree $2$, $\mathcal{I}$ is the square-free monomial ideal generated by $x_{i_1}\cdots x_{i_s}$ with $i_j$ distinct and $F_{i_1}\cap\cdots\cap F_{i_s} = \varnothing$ and $\mathcal{J}$ is the ideal generated by the linear forms $\sum_{i=1}^{m}\langle q, \eta(F_i)\rangle x_i$ where $q$ ranges over $\Lambda_r$.
\end{theorem}

\begin{theorem}[See, e.g., {\cite[section\ 12.4]{cox2024toric}}] \label{thm-pres-T-coh}
    Let $H^*_T(X(Y) ; \mathbb{Q})$ be the torus equivariant cohomology of $X(Y)$, then we have the following isomorphism of rings
\[
H^*_T(X(Y) ; \mathbb{Q})\simeq \mathbb{Q}[x_1, \ldots, x_m]/\mathcal{I}.
\]
\end{theorem}

The ring $\mathbb{Q}[x_1, \ldots, x_m]/\mathcal{I}$ is the well-known \emph{Stanley--Reisner ring} (or \emph{face ring}) associated with the simplicial fan $\Sigma_Y$ and it has a canonical $\mathbb{Q}$-vector space basis of the monomials supported on the cones of $\Sigma_Y$ \cite[pp.\ 3--5]{zbMATH02190625}.

\subsection{Toric orbifolds associated with dominant weight polytopes}

Note that in \cite{BGHpolytope}, the dominant weight polytope $P$ is showed to be a rational simple polytope, hence the associated toric variety $X(P)$ is an orbifold and we can apply Danilov's Theorem \ref{thm-pres-coh} to compute the cohomology of this toric orbifold.

The normal fan of $P$ is the complete simplicial fan given by the cones 
$$\sigma_{J, K}:=\operatorname{cone}\left(\left\{-\alpha_j^{\vee} \mid j \in J\right\} \cup\left\{\varpi_k^{\vee} \mid k \in K\right\}\right) \text { for disjoint subsets } J, K \subseteq [r].$$

We refer to \cite[sections 3 and 4]{abe2023peterson} for an explicit construction and a proof of projectivity of $X(P)$ and \cite{BGHpolytope} for the combinatorics of the dominant weight polytopes.

 \section{Proof of Main Theorems}

\subsection{Proof of Theorem \ref{thm-pres-P}}

Suppose $\Phi$ is an irreducible root system first. As $\lambda$ is strongly dominant, by the results in \cite{BGHpolytope}, the polytope $P$ is combinatorially equivalent to an $r$-dimensional cube. 
The $2r$ facets (that is, codimension-one faces) of $P$ are 
\[Q_i := \left\{p \in P \Biggm\vert p = \lambda - \sum_{k \ne i} a_k \alpha_k, a_k \in \mathbb{R}_{\ge 0} \right\}, \quad i \in [r],\]
and 
\[C_i := \left\{p \in P \Biggm\vert p = \sum_{k \ne i} b_k \varpi_k, b_k \in \mathbb{R}_{\ge 0} \right\}, \quad i \in [r].\]

Recall that the standard $r$-dimensional cube $[0,1]^r$ has $2r$ facets
\[H_{i,\varepsilon} := \left\{(a_1, \dots, a_r) \in [0,1]^r \mid a_i = \varepsilon \right\}, \quad i \in [r], \varepsilon \in \{0,1\}.\]
The aforementioned combinatorial equivalence between $P$ and $[0,1]^r$ sends $Q_i$ to $H_{i,1}$ and $C_i$ to $H_{i,0}$.

Recall that for $j=1, \dots, r$,  $\alpha_j\spcheck = \sum_{1 \le i \le r} \langle \alpha_i, \alpha_j\spcheck \rangle \varpi_i\spcheck$.
That is, the column vectors of the Cartan matrix $(\langle \alpha_i, \alpha_j\spcheck \rangle)_{i,j}$ of the  root system $\Phi$ are the coordinate vectors of $\alpha_j\spcheck$'s with respect to the basis $\varpi_1\spcheck, \dots, \varpi_r\spcheck$ of $\Lambda_r^*$.
By an examination of Cartan matrices (see, e.g., \cite[Plate I--IX]{Bourbaki-Lie456}), one can see that for an irreducible root system, all the simple coroots are primitive in the lattice $\Lambda_r^*$ except when $\Phi$ is of type $A_1$ or type $B_r$ ($r \ge 2$). 
In $A_1$ case, $\alpha_1\spcheck = 2 \varpi_1\spcheck$; 
In $B_r$ case, $\alpha_1\spcheck, \dots, \alpha_{r-1}\spcheck$ are primitive while $\alpha_r\spcheck = -2 \varpi_{r-1}\spcheck + 2 \varpi_r\spcheck$.
In other words, $\beta_r = \frac{1}{2} \alpha_r\spcheck$ if $\Phi$ is of type $A_1$ or $B_r$ ($r \ge 2$), and $\beta_i = \alpha_i\spcheck$ otherwise.

Let $S$ be a set of facets of $P$. 
The intersection $\bigcap_{F \in S} F$ is empty if and only if $Q_i, C_i \in S$ for some $i \in [r]$, that is, in $S$ there is a pair of facets in opposite positions. By Theorem \ref{thm-pres-T-coh}, we get the presentation of $H^*_T\left(X(P) ; \mathbb{Q}\right)$. Clearly, the outward primitive normal vector of $Q_i$ is $\varpi_i\spcheck$ and the outward primitive normal vector of $C_i$ is $-\beta_i$.
Therefore, by Theorem \ref{thm-pres-coh}, we know that Theorem \ref{thm-pres-P} holds if the root system is irreducible. In this presentation of $H^*\left(X(P) ; \mathbb{Q}\right)$, the variable $x_i$ corresponds to the facet $Q_i$, and $y_i$ to $C_i$.

The general case follows easily from the irreducible case as ``everything'' can be factored as a direct product, see, for example, \cite[p.\ 19--20]{fulton1993introduction}.
The proof of Theorem \ref{thm-pres-P} is complete.

\subsection{Combinatorial description of \texorpdfstring{$\mathit{WP}$}{WP} and \texorpdfstring{$H^*\left(X(\mathit{WP}) ; \mathbb{Q}\right)$}{H*(X(WP);Q)} by the Weyl group}

To begin with, for $i \in [r]$, let 
\[\widetilde{Q}_i : = \left\{p \in \mathit{WP} \Biggm\vert p = \lambda - \sum_{k \ne i} a_k \alpha_k, a_k \in \mathbb{R}_{\ge 0} \right\}.\]
Note that $\eta(\widetilde{Q}_j) = \eta(Q_j)=\varpi_i\spcheck$, as $\widetilde{Q}_j$ and $Q_j$ are contained in the same hyperplane.
Let $W_i$ be the maximal parabolic subgroup of $W$ excluding $s_i$ (the simple  reflection with respect to $\alpha_i$) and $W^i$ be the set of minimal representatives for the cosets $W / W_i$. We have the following lemma.

\begin{lemma} \label{lem-stab}
  \leavevmode
  \begin{enumerate}
      \item \label{lem-stab-1} The stabilizer of $\widetilde{Q}_i$ is $W_i$.
      \item \label{lem-stab-2} If $w \widetilde{Q}_i = \widetilde{Q}_j$ where $w \in W$ and $i,j \in [r]$, then $i = j$ and $w \in W_i$.
  \end{enumerate}
\end{lemma}

\begin{proof}
    The item \eqref{lem-stab-1} is a standard fact following from $\langle \alpha_i, \varpi_j\spcheck \rangle= \delta_{ij}$.

    If $w \widetilde{Q}_i = \widetilde{Q}_j$, then 
    \[w \varpi_i \spcheck = w \eta(\widetilde{Q}_i) = \eta(w \widetilde{Q}_i) = \eta(\widetilde{Q}_j) = \varpi_j\spcheck.\]
    This implies $i = j$ and $w \in W_i$.
    This proves \eqref{lem-stab-2}.
\end{proof}

By definition, the facets of $\mathit{WP}$ are those $\widetilde{Q}_i$'s and their $W$-translations.
More explicitly, it is easy to see from the above lemma that they are 
\[\{w \widetilde{Q}_i \mid i \in [r], w \in W^i\}.\]
The outward primitive normal vector of the facet $w \widetilde{Q}_i$ is $w \varpi_i\spcheck$.
By Theorem \ref{thm-pres-coh} again, we have

\begin{corollary}
    The ring $H^*\left(X(\mathit{WP}) ; \mathbb{Q}\right)$ has the following presentation
    \begin{equation} \label{eq-pres-WP}
      H^*\left(X(\mathit{WP}) ; \mathbb{Q}\right)=\mathbb{Q}\left[X_{i,w} \mid i \in [r], w \in W^i\right] / (\widetilde{\mathcal{I}}+ \widetilde{\mathcal{J}}),
    \end{equation}
    where $\widetilde{\mathcal{I}}$ is the ideal generated by
    \[\left\{ \prod_{1 \le j \le k} X_{i_j, w_j} \Biggm| k \in \mathbb{N}, i_j \in [r], w_j \in W^{i_j}, \bigcap_{j} w_{j} \widetilde{Q}_{i_j} = \emptyset \right\},\]
    and $\widetilde{\mathcal{J}}$ is the ideal generated by
    \[\left\{  \sum_{i \in [r], w \in W^i} \left\langle q, w \varpi_i\spcheck \right\rangle X_{i,w} \Biggm| q \in \Lambda_r \right\}.\]
\end{corollary}

In the above presentation of $H^*\left(X(\mathit{WP}) ; \mathbb{Q}\right)$, the variable $X_{i,w}$ corresponds to the facet $w \widetilde{Q}_i$.
\begin{notation} \label{nota}
    By abuse of notation, for an arbitrary element $w \in W$, we write $X_{i,w}$ to denote the variable $X_{i,\widetilde{w}}$ where $\widetilde{w} \in W^i \cap wW_i$ is the minimal element in $wW_i$.
\end{notation}

The $W$-action on $\mathit{WP}$ induces an action of $W$ on  $H^*\left(X(\mathit{WP}) ; \mathbb{Q}\right)$.
By definition, we have 
\[w' X_{i,w} = X_{i,w'w}.\]
Note that the ideals $\widetilde{\mathcal{I}}$ and $\widetilde{\mathcal{J}}$ are preserved by the $W$-action. 
For $\widetilde{\mathcal{I}}$, this is because \[\bigcap_{j} w'w_{j} \widetilde{Q}_{i_j} = w' \left(\bigcap_{j} w_{j} \widetilde{Q}_{i_j}\right) = \emptyset \] 
if $\bigcap_{j} w_{j} \widetilde{Q}_{i_j} = \emptyset$.
For $\widetilde{\mathcal{J}}$, this is because 
\[\sum_{i,w} \left\langle q, w \varpi_i\spcheck \right\rangle w'X_{i,w} = \sum_{i,w} \left\langle w'q, w'w \varpi_i\spcheck \right\rangle X_{i,w'w}\]
where $w'q \in \Lambda_r$ and $w'w$ can be replaced by the minimal element in $w'wW_i$.

Recall that $\lambda$ is strongly dominant and the vertices of $\mathit{WP}$ are $\{w \lambda \mid w \in W\}$. The following is a crucial lemma.

\begin{lemma} \label{lem-facet}
    \leavevmode
    \begin{enumerate}
      \item \label{lem-facet-1} If the vertex $\lambda$ lies in $w \widetilde{Q}_i$ for some $w \in W$ and some $i \in [r]$, then $\widetilde{Q}_i = w \widetilde{Q}_i$, that is, $w \in W_i$.
      \item \label{lem-facet-2} For any $w \in W$ and $i \in [r]$, if $\widetilde{Q}_i \ne w \widetilde{Q}_i$, then $\widetilde{Q}_i \cap w \widetilde{Q}_i = \emptyset$. 
    \end{enumerate}
\end{lemma}

\begin{proof}
    Suppose $\lambda \in w \widetilde{Q}_i$.
    Then, $w \widetilde{Q}_i = \widetilde{Q}_j$ for some $j \in [r]$, as $\widetilde{Q}_1, \dots, \widetilde{Q}_r$ are all the facets containing $\lambda$.
    By Lemma \ref{lem-stab}\eqref{lem-stab-2}, we have $i = j$ and $w \in W_i$.
    The item \eqref{lem-facet-1} is proved.

    Suppose now $\widetilde{Q}_i \cap w \widetilde{Q}_i \ne \emptyset$. 
    Then, there is a vertex of $\mathit{WP}$, say, $w'\lambda$, belonging to $\widetilde{Q}_i \cap w \widetilde{Q}_i$.
    Equivalently, $\lambda \in w'^{-1} \widetilde{Q}_i \cap w'^{-1} w \widetilde{Q}_i$.
    By \eqref{lem-facet-1}, this implies $\widetilde{Q}_i = w'^{-1} \widetilde{Q}_i = w'^{-1} w \widetilde{Q}_i$.
    By Lemma \ref{lem-stab}\eqref{lem-stab-1}, we have $w'^{-1}, w'^{-1} w \in W_i$.
    Therefore, $w \in W_i$ and hence $w \widetilde{Q}_i = \widetilde{Q}_i$.
    The item \eqref{lem-facet-2} is proved.
\end{proof}

\subsection{Proof of Theorem \ref{thm-main}}

Recall that $\eta(C_i) = -\beta_i$, $\eta(w\widetilde{Q}_i) = w \varpi_i\spcheck$, and $\{\eta(C_i) \mid i \in [r]\}$ is a basis for the space $E^*$.
For any $j \in [r]$ and $w \in W^j$, we have
\begin{equation} \label{eq-cijw}
    \eta(w \widetilde{Q}_j) - \eta(\widetilde{Q}_j) = \sum_{i \in [r]} c_{i,j,w} \eta(C_i) \text{ for some } c_{i,j,w} \in \mathbb{Q}.
\end{equation}

To prove Theorem \ref{thm-main}, we shall construct an isomorphism 
\[\varphi: H^*\left(X(P) ; \mathbb{Q}\right) \xrightarrow{\sim} H^*\left(X(\mathit{WP}) ; \mathbb{Q}\right)^W.\]
To this end, we use the presentations \eqref{eq-pres-P} and \eqref{eq-pres-WP} and define a ring homomorphism
\[\phi: \mathbb{Q}[x_1, \dots, x_r, y_1, \dots, y_r] \to H^*\left(X(\mathit{WP}) ; \mathbb{Q}\right)\]
by setting
\begin{equation} \label{eq-def-of-phi}
\begin{aligned}
    \phi(x_i) & := \sum_{w \in W^i} X_{i, w}, \\
    \phi(y_i) & := \sum_{j \in [r], w \in W^j} c_{i,j,w} X_{j,w}.
\end{aligned}
\end{equation}

\begin{proposition} \label{lem-im-inv}
    The image of $\phi$ is contained in $H^*\left(X(\mathit{WP}) ; \mathbb{Q}\right)^W$.
\end{proposition}

\begin{proof}
    It is clear that $\phi(x_i)$ belongs to $H^*\left(X(\mathit{WP}) ; \mathbb{Q}\right)^W$.

    Next we show that $\phi(y_i) \in H^*\left(X(\mathit{WP}) ; \mathbb{Q}\right)^W$.
    For each $i \in [r]$, fix a $q_i \in \Lambda_r$ such that $\langle q_i, \eta(C_j)\rangle = 0$ for all $j \ne i$ and $\langle q_i, \eta(C_i) \rangle \ne 0$. 
    For example, $q_i$ can be chosen to be $c \varpi_i$, where $c$ is the determinant of the Cartan matrix.
    In $H^*\left(X(\mathit{WP}) ; \mathbb{Q}\right)$ we have
    \begin{align}
        0 & = \sum_{j \in [r], w \in W^j} \left\langle q_i, \eta(w \widetilde{Q}_j) \right\rangle X_{j,w} \notag \\
        & = \sum_{j \in [r], w \in W^j} \left\langle q_i, \eta(\widetilde{Q}_j) + \sum_{k \in [r]} c_{k,j,w} \eta(C_k) \right\rangle X_{j,w} \notag \\
        & = \sum_{j \in [r], w \in W^j} \left\langle q_i, \eta(\widetilde{Q}_j) +  c_{i,j,w} \eta(C_i) \right\rangle X_{j,w} \notag \\
        & = \sum_{j \in [r], w \in W^j} \left\langle q_i, \eta(\widetilde{Q}_j) \right\rangle X_{j,w} + \left\langle q_i,  \eta(C_i) \right\rangle\sum_{j \in [r], w \in W^j} c_{i,j,w} X_{j,w}. \label{eq-lem-im-inv-1}
    \end{align}
    Note that the first term in \eqref{eq-lem-im-inv-1} is $W$-invariant. 
    Therefore, the second term, which is a nonzero multiple of $\phi(y_i)$, is also $W$-invariant. 
\end{proof}

\begin{proposition} \label{lem-ker-IJ}
    The kernel of $\phi$ contains the ideals $\mathcal{I}$ and $\mathcal{J}$.
\end{proposition}

\begin{proof}
    First, we show that $\mathcal{I} \subseteq \ker \phi$, that is, 
    \[\phi(x_i y_i) = 0\]
    for each $i \in [r]$.
    By Proposition \ref{lem-im-inv}, $\phi(y_i)$ is $W$-invariant.
    Thus we have 
    \begin{align*}
        \phi(x_i y_i) & = \left( \sum_{w \in W^i} X_{i,w}  \right) \cdot \phi(y_i) \\
        & = \sum_{w \in W^i} X_{i,w} \cdot \left( w \phi(y_i) \right) \\
        & = \sum_{w \in W^i} w \left( X_{i,e} \phi(y_i) \right) \\
        & = \sum_{w \in W^i} w \left( \sum_{j \in [r], w \in W^j} c_{i,j,w} X_{i,e} X_{j,w} \right).
    \end{align*}
    We claim that
    \begin{equation*} 
        \text{if } X_{i,e} X_{j,w} \ne 0, \text{ or equivalently, if } \widetilde{Q}_i \cap w \widetilde{Q}_j \ne \emptyset, \text{ then } c_{i,j,w} = 0
    \end{equation*}
    (and thus $\phi(x_i y_i) = 0$).
    Suppose $\widetilde{Q}_i \cap w \widetilde{Q}_j \ne \emptyset$.
    Then, there is a vertex of $\mathit{WP}$, say, $w' \lambda$ ($w' \in W$),  belonging to  $\widetilde{Q}_i \cap w \widetilde{Q}_j$.
    Note that $\lambda$ is a vertex of $\widetilde{Q}_i$ and $w'^{-1} \widetilde{Q}_i$ simultaneously. 
    Thus, $w' \in W_i$ by Lemma \ref{lem-facet}\eqref{lem-facet-1}.
    Moreover, we have $w' \lambda \in w \widetilde{Q}_j \cap w' \widetilde{Q}_j$.
    Therefore, $w \widetilde{Q}_j = w' \widetilde{Q}_j$ by Lemma \ref{lem-facet}\eqref{lem-facet-2}.
    Consequently, we have
    \begin{equation} \label{eq-lem-ker-IJ-2}
        \begin{split}
            \eta(w \widetilde{Q}_j) - \eta(\widetilde{Q}_j) & = \eta(w' \widetilde{Q}_j) - \eta(\widetilde{Q}_j) \\
            & = w' \eta(\widetilde{Q}_j) - \eta(\widetilde{Q}_j) \\
            & \in \sum_{k \ne i} \mathbb{R} \alpha_k\spcheck = \sum_{k \ne i} \mathbb{R} \eta(C_k).
        \end{split}
    \end{equation}
    Comparing Equations \eqref{eq-lem-ker-IJ-2} and  \eqref{eq-cijw}, we see that $c_{i,j,w} = 0$ as claimed.

    It remains to show $\mathcal{J} \subseteq \ker \phi$.
    For this, we need to show that for any $q \in \Lambda_r$, 
    \begin{equation} \label{eq-lem-ker-IJ-1}
        \phi \left( \sum_{i\in [r]}\left\langle q, \eta(Q_i) \right\rangle x_i+\sum_{i\in [r]}\left\langle q, \eta(C_i) \right\rangle y_i \right) = 0.
    \end{equation}
    By definition, we have
    \begin{align*}
        \text{Left-hand side of } \eqref{eq-lem-ker-IJ-1} & = \sum_{i\in [r]} \left\langle q, \eta(Q_i) \right\rangle \sum_{w \in W^i} X_{i,w} + \sum_{i\in [r]} \left\langle q, \eta(C_i) \right\rangle \sum_{j \in [r], w \in W^j} c_{i,j,w} X_{j,w}\\
        & = \sum_{j \in [r], w \in W^j}  \left\langle q, \eta(Q_j) + \sum_{i\in [r]} c_{i,j,w} \eta(C_i) \right\rangle  X_{j,w} \\
        & = \sum_{j \in [r], w \in W^j} \left\langle q, \eta(w \widetilde{Q}_j) \right\rangle  X_{j,w} \in \widetilde{\mathcal{J}}.
    \end{align*}
    Thus, Eq.~\eqref{eq-lem-ker-IJ-1} holds.
\end{proof}

\begin{proposition} \label{lem-im=inv}
    The image of $\phi$ is exactly $H^*\left(X(\mathit{WP}) ; \mathbb{Q}\right)^W$.
\end{proposition}

\begin{proof}
    Recall that the $T$-equivariant cohomology ring $H^*_T\left(X(\mathit{WP}) ; \mathbb{Q}\right)$ has a presentation 
    \begin{equation*} 
      H^*_T\left(X(\mathit{WP}) ; \mathbb{Q}\right)=\mathbb{Q}\left[X_{i,w} \mid i \in [r], w \in W^i\right] / \widetilde{\mathcal{I}}
    \end{equation*}
    (see Theorem \ref{thm-pres-T-coh}).
    We have a natural projection 
    \[p : H^*_T\left(X(\mathit{WP}) ; \mathbb{Q}\right) \twoheadrightarrow H^*\left(X(\mathit{WP}) ; \mathbb{Q}\right).\]
    This projection is $W$-equivariant (that is, a homomorphism of $W$-modules), as the ideal $\widetilde{\mathcal{J}}$ is preserved by the $W$-action.
    Therefore, the restriction to the $W$-fixed part $H^*_T\left(X(\mathit{WP}) ; \mathbb{Q}\right)^W \to H^*\left(X(\mathit{WP}) ; \mathbb{Q}\right)^W$ is also surjective by a standard averaging argument.
    Then, this proposition follows from the following Lemma \ref{lem-T-coh-gen}, which implies that $H^*\left(X(\mathit{WP}) ; \mathbb{Q}\right)^W$ is generated by all the $\phi(x_i)$'s.
\end{proof}

We have the following lemma, whose proof leverages the Weyl group action on the face structure of $\mathit{WP}$.
\begin{lemma} \label{lem-T-coh-gen}
    The algebra $H^*_T(X(\mathit{WP}) ; \mathbb{Q})^W$ is generated by 
    \[\sum_{w \in W^i} X_{i,w}, \quad i \in [r].\]
\end{lemma}

\begin{proof}
    The elements in $H^*_T(X(\mathit{WP}) ; \mathbb{Q})$ are polynomials in the variables $X_{i,w}$.
    Suppose $f \in H^*_T(X(\mathit{WP}) ; \mathbb{Q})^W$ and $f \ne 0$.
    Let $X_{i_1, w_1}^{n_1} \cdots X_{i_k, w_k}^{n_k}$ be a nonzero monomial appearing in $f$, that is,
    \[f = c X_{i_1, w_1}^{n_1} \cdots X_{i_k, w_k}^{n_k} + \text{ other terms, with $c \ne 0$}.\]
    Then by definition, $w_1 \widetilde{Q}_{i_1} \cap \dots \cap w_k \widetilde{Q}_{i_k} \ne \emptyset$.
    There is a vertex of $\mathit{WP}$, say, $w_0\lambda$, lying in this intersection.
    Then, 
    \[\lambda \in w_0^{-1} w_1 \widetilde{Q}_{i_1} \cap \dots \cap w_0^{-1} w_k \widetilde{Q}_{i_k}.\] 
    By Lemma \ref{lem-facet}\eqref{lem-facet-1}, this implies $w_0^{-1} w_j \widetilde{Q}_{i_j} = \widetilde{Q}_{i_j}$ for each $j$. 
    In particular,
    \[w_0^{-1} X_{i_1, w_1}^{n_1} \cdots X_{i_k, w_k}^{n_k} = X_{i_1, e}^{n_1} \cdots X_{i_k, e}^{n_k}.\]
    As $f$ is $W$-invariant, the monomial $X_{i_1, e}^{n_1} \cdots X_{i_k, e}^{n_k}$ also appears in $f$ and we can assume $i_1, \cdots, i_k$ are mutually distinct.

    Let $I := [r] \setminus \{i_1, \dots, i_k\}$, $W_I := \langle s_i \mid i\in I \rangle$ (the parabolic subgroup of $W$ corresponding to $I$), and $W^I$ be the set of minimal representatives for left cosets $W / W_I$.
    Then,
    \[\operatorname{Stab}(X_{i_1, e}^{n_1} \cdots X_{i_k, e}^{n_k}) = \bigcap_{1 \le j \le k} \operatorname{Stab}(X_{i_j,e}) = \bigcap_{1 \le j \le k} W_{i_j} = W_I.\]
    The first equality is due to the fact that $\widetilde{Q}_i$ and $\widetilde{Q}_j$ belong to different $W$-orbit whenever $i \ne j$ (see Lemma \ref{lem-stab}\eqref{lem-stab-2}).
    
    Because $f$ is $W$-invariant, all the $W$-translations of $X_{i_1, e}^{n_1} \cdots X_{i_k, e}^{n_k}$ must appear in $f$ with the same coefficient.
    Therefore, 
    \[f = c \sum_{w \in W^I} X_{i_1, w}^{n_1} \cdots X_{i_k, w}^{n_k} + \text{ other terms, with $c \ne 0$}.\]
    By induction, it suffices to prove the following equality in $H^*_T(X(\mathit{WP}) ; \mathbb{Q})$ :
    \begin{equation} \label{eq-claim-2}
        \sum_{w \in W^I} X_{i_1, w}^{n_1} \cdots X_{i_k, w}^{n_k} = \left( \sum_{w \in W^{i_1}} X_{i_1,w} \right)^{n_1} \cdots \left( \sum_{w \in W^{i_k}} X_{i_k,w} \right)^{n_k}.
    \end{equation}
    By Lemma \ref{lem-stab}\eqref{lem-stab-2}, $w \widetilde{Q}_i$ and $w' \widetilde{Q}_i$ have empty intersection whenever $w \ne w'$.
    Thus, we have $(\sum_{w \in W^i} X_{i,w})^n = \sum_{w \in W^i} X_{i,w}^n$, and Eq.~\eqref{eq-claim-2} is equivalent to 
    \begin{equation} \label{eq-lem-T-coh-gen}
        \sum_{w \in W^I} X_{i_1, w}^{n_1} \cdots X_{i_k, w}^{n_k} =  \sum_{w_1 \in W^{i_1}, \dots, w_k \in W^{i_k}} X_{i_1,w_1}^{n_1}  \cdots  X_{i_k,w_k}^{n_k}.
    \end{equation}
    Note that the summands in the left-hand side of Eq.~\eqref{eq-lem-T-coh-gen} are different to each other, that is, appear only once.
    On the other hand, each summand in the right-hand side of Eq.~\eqref{eq-lem-T-coh-gen}, if nonzero, appears only once as well.
    Therefore, it suffices to show that all the summands in the left-hand side appear in the right-hand side, and vice versa.

    For any $w \in W^I$ and $j \in [k]$, let $w_j$ denote the minimal representative in the coset $wW_{i_j}$.
    Then, the summand $X_{i_1, w}^{n_1} \cdots X_{i_k, w}^{n_k}$ in the left-hand side of Eq.~\eqref{eq-lem-T-coh-gen} equals $X_{i_1, w_1}^{n_1} \cdots X_{i_k, w_k}^{n_k}$, which appears in the right-hand side of Eq.~\eqref{eq-lem-T-coh-gen}.
    Conversely, suppose $X_{i_1, w_1}^{n_1} \cdots X_{i_k, w_k}^{n_k}$ ($w_j \in W^{i_j}$) is a nonzero monomial.
    Then, as above, there is a vertex of $\mathit{WP}$, say, $w' \lambda$, belonging to $w_1 \widetilde{Q}_{i_1} \cap \cdots \cap w_k \widetilde{Q}_{i_k}$.
    Equivalently, we have
    \[\lambda \in (w')^{-1} w_1 \widetilde{Q}_{i_1} \cap \cdots \cap (w')^{-1} w_k \widetilde{Q}_{i_k}.\]
    By Lemma \ref{lem-facet}\eqref{lem-facet-1} again, we have $w_j \in w'W_{i_j}$ for each $j$, and hence $X_{i_j,w_j} = X_{i_j, w'}$.
    Let $w \in W^I$ be the minimal representative in  the coset $w' W_I$.
    Then, 
    \[X_{i_1, w_1}^{n_1} \cdots X_{i_k, w_k}^{n_k} = X_{i_1, w'}^{n_1} \cdots X_{i_k, w'}^{n_k}= X_{i_1, w}^{n_1} \cdots X_{i_k, w}^{n_k},\]
    which is a summand in the left-hand side of Eq.~\eqref{eq-lem-T-coh-gen}.
    The proof is complete.
\end{proof}

By Propositions \ref{lem-ker-IJ} and \ref{lem-im=inv}, the homomorphism $\phi$ induces a surjective homomorphism
\[\varphi: H^*\left(X(P) ; \mathbb{Q}\right) \twoheadrightarrow H^*\left(X(\mathit{WP}) ; \mathbb{Q}\right)^W.\]
The following lemma indicates that $\varphi$ is in fact an isomorphism.

\begin{lemma} \label{lem-dim}
    $\dim H^*\left(X(P) ; \mathbb{Q}\right) = 2^r = \dim H^*\left(X(\mathit{WP}) ; \mathbb{Q}\right)^W$.
\end{lemma}

\begin{proof}
    As the odd-degree cohomology of $X(P)$ vanishes (see Theorem \ref{thm-pres-coh}), the leftmost term $\dim H^*\left(X(P) ; \mathbb{Q}\right)$ equals the Euler characteristic of $X(P)$. 
    It is well-known that the Euler characteristic of a toric variety $X(P)$ equals the number of full dimensional cones in the defining fan (see \cite[sections 11 and 12]{danilov1978geometry} or \cite[sections 3.2 and 4.5]{fulton1993introduction}). 
    This also equals the number of vertices of the polytope $P$, which is $2^r$ as $P$ is combinatorially equivalent to the $r$-dimensional cube. 
    This proves the first equality. 
    For the second equality, see \cite[Theorem 3.5(ii)]{lehrer2008rational}.
\end{proof}

\begin{remark}
    One can also use the fact that $H^*\left(X(P) ; \mathbb{Q}\right)$ is a Poincar\'e duality algebra to show that the surjection $\phi$ is indeed an isomorphism (cf., \cite{HMSS21-toric-orbifolds, Song22}).
    See, for example, \cite[Definition 10.4, Lemma 10.5]{abe19cohomology}.
\end{remark}

From the definition of $\phi$, one sees that the isomorphism $\varphi$ preserves the grading.
The proof of Theorem \ref{thm-main} is complete.

\section{A generalization to finite Coxeter groups} \label{sec-gen}

Let $W$ be a finite Coxeter group with the finite set $S=\{s_1, \cdots, s_r\}$ of simple reflections, that is, 
\[W=\left\langle s \in S \mid (s t)^{m_{s t}}=\mathrm{id} \text { for any } s, t \in S \text { with } m_{s t}<\infty\right\rangle,\]
where $m_{s s}=1$ for each $s \in S$, and $m_{s t}=m_{t s}$. It follows that $m_{s t}$ is precisely the order of the element $s t$. It is known that all finite Weyl groups have a structure of Coxeter group, but not all finite Coxeter groups can be realized as the Weyl group of some Lie algebra. For example, among the irreducible Coxeter groups, $H_3$, $H_4$, and dihedral group $I_2(m)$ for $m \ne 3, 4, 6$ are not Weyl groups (see \cite{Humphreys90}).

Any Coxeter group $W$ has a geometric representation on $V:=\bigoplus_{s \in S}\mathbb{R}\alpha_s$, which is defined by the action of each simple reflection $s \in S$,
$$
s(\lambda)=\lambda-2\left(\lambda, \alpha_s\right) \alpha_s , \quad \forall \lambda \in V,
$$
where the symmetric bilinear form $(-,-)$ on $V$ is determined by
$$
\left(\alpha_s, \alpha_t\right)=-\cos \frac{\pi}{m_{s t}}.
$$
Note that $\left(\alpha_s, \alpha_s\right)=1$. It follows that the symmetric bilinear form $(-,-)$ on $V$ is $W$-invariant for the geometric representation. It is known that for finite Coxeter group $W$,  the bilinear form $(-,-)$ is positive definite, hence it defines a $W$-invariant inner product on $V$. It is also know that the geometric representation is faithful for any Coxeter group, thus realizing a finite Coxeter group as a finite reflection group on a Euclidean space. We denote the dual basis of $\{\alpha_s, s\in S\}$ under the inner product $(-,-)$ as $\varpi_1, \dots, \varpi_r$. That is, $(\alpha_{s_i}, \varpi_j)=\delta_{ij}$, where $\delta_{ij}$ is the Kronecker delta.

We maintain the notation of $W_i$ to be the maximal parabolic subgroup of $W$ excluding $s_i$ and $W^i$ be the set of minimal representatives for the cosets $W / W_i$. 

Define the graded $\mathbb{R}$-algebra $R\left(P\right)$ to be
    \begin{equation} \label{eq-def-R(P)}
     R\left(P\right)=\mathbb{R}\left[x_1, \ldots, x_r, y_1, \ldots, y_r\right] /\left(\mathcal{I}+\mathcal{J}\right),
    \end{equation}
    where $\mathcal{I}$ is the ideal generated by
    \[\{x_i y_i \mid i \in [r]\},\]
    and $\mathcal{J}$ is the ideal generated by 
    \[\left\{\sum_{i=1}^r(\alpha_{s_j},  \varpi_i)x_i+\sum_{i=1}^r(\alpha_{s_j}, -\alpha_{s_i} ) y_i \Biggm\vert j \in [r]\right\}=\left\{x_j-\sum_{i=1}^r(\alpha_{s_j}, \alpha_{s_i} ) y_i \Biggm\vert j \in [r]\right\}.\]
    
Define the graded $\mathbb{R}$-algebra $R\left(\mathit{WP}\right)$ to be
    \begin{equation} \label{eq-def-R(WP)}
      R\left(\mathit{WP}\right)=\mathbb{R}\left[X_{i,w} \mid i \in [r], w \in W^i\right] / (\widetilde{\mathcal{I}}+ \widetilde{\mathcal{J}}),
    \end{equation}
    where $\widetilde{\mathcal{I}}$ is the ideal generated by
    \[\left\{ \prod_{1 \le j \le k} X_{i_j, w_j} \Biggm| k \in \mathbb{N}, i_j \in [r], w_j \in W^{i_j}, \bigcap_{j} w_{j} \widetilde{Q}_{i_j} = \emptyset \right\},\]
    and $\widetilde{\mathcal{J}}$ is the ideal generated by
    \[\left\{  \sum_{i \in [r], w \in W^i} (\alpha_j, w \varpi_i) X_{i,w} \Biggm| j \in [r]\right\}.\]

\begin{remark} \label{rmk-cry}
    When $W$ is a Weyl group, one can define a representation of $W$ over $\mathbb{Z}$ rather than over $\mathbb{R}$ by modifying the $\alpha_s$'s in the geometric representation with possibly different lengths (e.g., simple roots of non-simply-laced type Lie algebras). This integral representation can be realized on the Cartan subalgebra of the corresponding complex semisimple Lie algebra using the Cartan matrix. However, as $W$-representations, this representation is isomorphic to the geometric representation over $\mathbb{R}$. 
    It is not difficult to see that in this case, the above graded rings are isomorphic to base change of $H^*\left(X(P) ; \mathbb{Q}\right)$ and $H^*\left(X(\mathit{WP}) ; \mathbb{Q}\right)$, respectively. However, if a finite Coxeter group $W$ is not a Weyl group, then there is no $W$-invariant lattice in the geometric representation $V$. Consequently, there are no corresponding toric variety in this case, although these rings ``behave as the cohomology of a toric variety''. For example, McMullen \cite{mcmullen1989polytope} proved the polytope algebra of any complete simplicial fan (which is isomorphic to $R\left(P\right)$ and $R\left(\mathit{WP}\right)$ for the normal fan of $P$ and $\mathit{WP}$, respectively) satisfies the analog of the hard Lefschetz and Hodge--Riemann bilinear relations in Hodge theory for rationally smooth projective varieties, while Brion \cite{brion1997structure} proved these polytope algebras are Poincar\'e duality algebras and studied the Hodge structures of them.
\end{remark}

\begin{theorem} \label{thm-noncry}
   Let $\Phi$ be a reduced (not necessarily crystallographic) root system, let $W$ be the finite Coxeter group generated by the simple reflections along the simple roots, and let $P$ be the dominant weight polytope associated with any strongly dominant weight. Then there exists an explicit graded ring isomorphism
   \begin{equation} 
     R\left(P\right) \cong R\left(\mathit{WP}\right)^W.      
   \end{equation}

\end{theorem} 

\begin{proof}
    The proof of Theorem \ref{thm-main} works for noncrystallographic root systems with suitable minor modifications. Here we only point out necessary modifications. 
    
    First, the $c_{i,j,w} \in \mathbb{Q}$ in Eq.~\eqref{eq-cijw} need be replaced by $c_{i,j,w} \in \mathbb{R}$. Then one can construct the map $\widetilde{\phi}: R\left(P\right) \longrightarrow R\left(\mathit{WP}\right)^W$ with the same definition in \eqref{eq-def-of-phi}. The proof of that the image of $\widetilde{\phi}$ is contained in $R\left(\mathit{WP}\right)^W$ is the same as the proof of Proposition \ref{lem-im-inv}. The proof of the kernel of $\widetilde{\phi}$ contains the ideals $\widetilde{\mathcal{I}}$ and $\widetilde{\mathcal{J}}$ is the same as the proof of Proposition \ref{lem-ker-IJ}. The proof of that the image of $\widetilde{\phi}$ is exactly $R\left(\mathit{WP}\right)^W$ is the same as the proof of Proposition \ref{lem-im=inv} because one can easily check that the proof of Lemma \ref{lem-T-coh-gen} works for the Stanley--Reisner ring $R\left(\mathit{WP}\right)=\mathbb{R}\left[X_{i,w} \mid i \in [r], w \in W^i\right] / \widetilde{\mathcal{I}}$. Finally we have an analog of Lemma \ref{lem-dim}, that is, 
     $$\dim R\left(P\right) = 2^r =\dim R\left(\mathit{WP}\right)^W.$$
     The first equality follows from a simple computation of the $h$-vector of the simplicial normal fan of $P$ and the fact that $P$ is combinatorially equivalent to the $r$-dimensional cube, while the second one is proved in \cite[p.\ 258]{stembridge1994some}. More precisely, both    $R\left(P\right)$ and  $R\left(\mathit{WP}\right)^W$ has the same Poincar\'e polynomial $(1+q)^r$. The proof is completed. 
\end{proof} 

\begin{remark}
    For a comparison, note that the Poincar\'e polynomial of $R\left(\mathit{WP}\right)$, which is the Poincar\'e polynomial of the cohomology of the toric variety $X(\mathit{WP})$ in the crystallographic case, is the $h$-polynomial of $\mathit{WP}$,
    given by
$$
h_W(q)=\sum_{w \in W} q^{\operatorname{des}(w)},
$$
where $\operatorname{des}(w)$ denotes the number of (right) descents of the element $w$. See, for example, \cite{zbMATH03869390}. This polynomial is known as the $W$-Eulerian polynomial, which does not only depend on the rank of $W$.

\end{remark}

\bibliographystyle{amsplain}
\bibliography{toric-symm}

\end{document}